\documentclass[10pt,reqno]{amsart}
\usepackage{amsmath,amssymb}
\usepackage[arrow,matrix,curve]{xy}

\usepackage{mathrsfs,bbm,eucal}
\theoremstyle{plain}
\newtheorem{thm}{Theorem}
\theoremstyle{plain}

\theoremstyle{plain}
\newtheorem{prop}[thm]{Proposition}
\theoremstyle{plain}

\newtheorem{cor}[thm]{Corollary}
\theoremstyle{definition}
\newtheorem{defn}[thm]{Definition}
\theoremstyle{remark}
\newtheorem{rem}[thm]{Remark}

\DeclareMathAlphabet{\mathpzc}{OT1}{pzc}{m}{it}

\newcommand{\dirac}{\mbox{$\mathcal{D}\!\!\!\!\!\:/\!\;$}}
\newcommand{\spinor}{\mbox{$S\!\!\!\!\!\:/\;\!$}}
\newcommand{\bundle}[1]{\CMcal{#1}}

\newcommand{\R}{\mathbbm{R}}
\newcommand{\C}{\mathbbm{C}}

\newcommand{\Q}{\mathbbm{Q}}
\renewcommand{\H}{\mathscr{H}}
\newcommand{\He}{\mathscr{H}^\partial }
\newcommand{\Z}{\mathbbm{Z}}
\newcommand{\N}{\mathbbm{N}}

\newcommand{\ke}{\mathpzc{k}}
\newcommand{\BU}{\mathrm{B}\mathrm{U}}
\newcommand{\group}{G}

\begin{document}
\title{Homology of finite K--area}
\author{Mario Listing}
\address{Mathematisches Institut, Albert-Ludwigs-Universit\"at Freiburg, Eckerstra\ss e 1, 79104 Freiburg, Germany}
\email{mario.listing@math.uni-freiburg.de}
\thanks{Supported by the German Science Foundation and in Part by SFB/TR 71}
\begin{abstract}
We use Gromov's K--area to define a generalized homology theory on compact smooth manifolds. In fact, this theory collects obstructions to the enlargeability of the manifold and its nontrivial submanifolds. Moreover, using the K--area homology we can rephrase some classic results about positive scalar curvature.
\end{abstract}
\keywords{K--area, enlargeable manifolds, homology theory, scalar curvature}
\subjclass[2000]{55N20 (57R20,57R19,53C21)}
\maketitle

\section{Introduction}
Gromov introduced in \cite{Gr01} the notion of K--area for Riemannian manifolds and proved that this area is finite for closed spin manifold of positive scalar curvature. In this result the K--area replaces the concept of enlargeability considered by Gromov and Lawson in \cite{GrLa2,GrLa3}. Apart from this interesting fact, the K--area is not  entirely bound to positive scalar curvature, because finite K--area of a compact manifold depends only on the homotopy type of the manifold whereas the existence of positive scalar curvature depends on the differentiable structure. So far the best obstruction for positive scalar curvature on spin manifolds is the vanishing of an invariant $\alpha ^\R _\mathrm{max}(M)\in \mathrm{KO}_n(C^*_{\mathrm{max},\R }\pi _1(M))$ introduced by Rosenberg in \cite{Ros1,Ros2}. In particular, $\alpha ^\R _\mathrm{max}(M)$ generalizes the Atiyah--Milnor--Singer invariant and does not vanish for enlargeable spin manifolds which was recently proved by Hanke and Schick in \cite{HaSch1,HaSch2}. The advantage of the K--area is the intrinsic definition, but also its behaviour under topological methods. Moreover, concerning positive scalar curvature it is an open question  if the case of infinite K--area on closed spin manifolds is covered by the condition $\alpha ^\R _\mathrm{max}(M)\neq 0$. 

We generalize Gromov's definition and consider the K--area of a compact Riemannian manifold $(M,g)$ w.r.t.~a homology class $\theta \in H_{*}(M;G)$ where $H_*(M;G)$ means singular homology of $M$ with coefficients in an abelian group $G$. This leads to the definition of the homology groups with finite K--area respectively the K--area homology. In fact, the set of homology classes $\theta \in H_k(M;G)$ with finite K--area determine a subgroup $\H _k(M;G)\subset H_k(M;G)$ which is independent on the choice of the Riemannian metric on $M$. Moreover, the induced homomorphism of a continuous map $f:M\to N$ restricts to homomorphisms $f_*:\H _k(M;G)\to \H _k(N;G)$ which proves that $\H _k(M;G)$ depends only on the homotopy type of the compact manifold $M$. In the paper we give some examples of the K--area homology which show its nontrivial character. For instance, compact orientable surfaces of positive genus and tori have trivial K--area homology whereas $\H _{2k}(M)=H_{2k}(M)$ for all $k>0$ if $M$ is a closed orientable manifold with finite fundamental group. Because the fundamental class of the sphere $S^k $ has finite K--area for all $k\geq 2$ we conclude as an application that the image of the Hurewicz homomorphism $h_k:\pi _k(M)\to H _k(M)$ is contained in $\H _k(M)$ for all $k\geq 2$. Conversely, the fundamental class of a connected compactly enlargeable manifold $M^n$ has infinite K--area which means $\H _n(M)=\{ 0\} $.  Moreover, if $M$ is connected and compactly $\widehat{A}$--enlargeable, then the Poincare dual of the total $\widehat{A}$--class of $M$ has infinite K--area, i.e.~$\widehat{\mathbf{A}}(M)\cap [M]\notin \H _*(M;\Q )$.
Here, compactly enlargeable respectively compactly $\widehat{A}$--enlargeable refer to finite coverings in the definition of enlargeability (cf.~\cite{LaMi}). Moreover, we show that a closed connected spin manifold $M^n$ of positive scalar curvature satisfies $\widehat{\mathbf{A}}(M)\cap [M]\in \H _*(M;\Q )$ and $\H _{n-1}(M)=H _{n-1}(M)$ which together with the above observation can be seen as a restatement of some classic results about positive scalar curvature. Note that $\H _0(M;\Q )=\H _1(M;\Q )=\{ 0\} $ for any compact manifold, i.e.~this result covers the original obstruction of the vanishing $\widehat{A}$--genus.  In the last section we introduce the relative K--area of a class $\theta \in H_k(M,A)$ where $A$ is a compact submanifold of $M$. This leads to the subgroups $\H _k(M,A)\subset H_k(M,A)$ of homology classes with finite relative K--area. In fact, this relative version satisfies the excision property and $\H _k(M,A)$ is isomorphic to $\H _k(M/A)$ for all $k$ if $M/A$ is a smooth manifold. Hence, the subfunctor $\H \subset H$ determines a generalized homology theory on pairs of compact smooth manifolds and continuous maps which satisfies the Eilenberg--Steenrod axioms up to exactness. In the end  we give an extension of the K--area homology to pairs of topological spaces. Although the K--area homology puts some classic results about positive scalar curvature in an interesting setting, it can obviously not cover ''exotic'' index theoretic obstructions like  the nonvanshing of $\alpha _\mathrm{max}^\R (M)$ on certain exotic spheres. However, as a generalized homology theory it contains topolocigal data which can not be seen by singular homology. 

Note that Brunnbauer and Hanke introduced in \cite{pre_BrHa} the small group homology which is closely related to our K--area homology. In fact, using the results in this paper and in \cite{pre_BrHa} one can show that $\H _*(M;\Q )$ is a subspace of the small group homology $H_*^{\mathrm{sm}(P)}(M;\Q )$ if $P$ denotes the largeness property ''compactly enlargeable''. 
\section{The K--area of a homology class}
Let $M_g:=(M,g)$ be a compact Riemannian manifold possibly with nonempty boundary. In order to obtain the additivity axiom  and in view of the topological K--theory the fiber dimension of a vector bundle on $M$ is not assumed to be a global constant, it is only constant on connected components. Suppose that $\mathrm{BU}_n$ is the classifying space of $\mathrm{U}(n)$, then $n$--dimensional Hermitian vector bundles on $M$ are classified up to isomorphism by $[M,\mathrm{BU}_n]$.  Hence, $[M,\BU ] $ classifies Hermitian vector bundles on $M$ up to isomorphism where $\BU =\coprod _n\mathrm{BU}_n$  is the disjoint (topological) sum. If $\bundle{E}\to M$ is a (smooth) Hermitian vector bundle with Hermitian connection, we denote by $\rho ^\bundle{E}:M\to \BU $ the classifying map and define
\[
\| R^\bundle{E}\| _g:=\max _{v, w\in TM}\frac{|R^\bundle{E}(v\wedge w)|_{op}}{|v\wedge w |_g}
\] 
where $R^\bundle{E}:\Lambda ^2TM\to \mathrm{End}(\bundle{E})$ means the curvature of $\bundle{E}$ and $|.| _{op}$ is the operator norm on $\mathrm{End}(\bundle{E})$. It is quite essential for the theory to use the operator norm on $\mathrm{End}(\bundle{E})$ because the equivalence of norms on finite dimensional vector spaces includes a constant usually depending on the dimension and in case of infinite K--area, the rank of the interesting bundles tend to infinity. Suppose $\theta \in H_{2*}(M;\group )$ for a coefficient group $\group $ where omitting the coefficient group means as usual $G=\Z $. Then $\mathscr{V}(M;\theta )$ denotes the set of all Hermitian bundles $\bundle{E}\to M$ endowed with a Hermitian connection such that the image of $\theta $ under the induced homomorphism $\rho ^\bundle{E}_*:H_*(M;\group )\to H_*(\BU ;\group )$ is nontrivial: $\rho ^\bundle{E}_*(\theta )\neq 0$. Since the classifying map is uniquely determined up to homotopy, $\mathscr{V}(M;\theta )$ does not depend on the choice of $\rho ^\bundle{E}$. The $\Z $--cohomology ring of $\mathrm{BU}_n$ is a $\Z $--polynomial ring  generated by the Chern classes which supplies the following alternative definition of $\mathscr{V}(M;\theta )$. Suppose that $M$ is connected, $2n\geq \dim M$ and $\theta \in H_{2*}(M)$, then $\mathscr{V} (M;\theta )$ is the set of Hermitian bundles $\bundle{E}\to M$ endowed with a Hermitian connection such that there is a polynomial $p\in \Z [c_1,\ldots ,c_n] $ with $\left< p(c(\bundle{E})),\theta \right> \neq 0$ where $c(\bundle{E})$ is the total Chern class of $\bundle{E}$ and
\[
p(c(\bundle{E})):=p(c_1(\bundle{E}),\ldots ,c_n(\bundle{E}))\in H^{2*}(M;\Z  ).
\]
Of course, if $\bundle{E}\in \mathscr{V} (M;\theta )$, then $\left< p(c(\bundle{E})),\theta \right> \neq 0$ for a monomial $p$, i.e.~there is a nonvanishing $\theta $--Chern number:
\[
\left< c_{1}(\bundle{E})^{j_1}\cdots c_{n}(\bundle{E})^{j_n},\theta \right> \neq 0
\]
for certain nonnegative integers $j_1\ldots ,j_n$. The equivalence of the two descriptions is a simple exercise in algebraic topology because $H^*(\mathrm{BU}_m;\Z )=\Z [c_1,\ldots ,c_m]$, $m=\mathrm{rk}_\C (\bundle{E})$, yields the nondegeneracy of the pairing $\left< .,.\right> $ and $\rho ^\bundle{E}_*(\theta )\neq 0$ implies the existence of a characteristic class $u\in H^*(\mathrm{BU}_m;\Z )$ with $\left< u,\rho ^\bundle{E}_*(\theta )\right> \neq 0$, i.e.~we choose $p(c(\bundle{E}))=(\rho ^\bundle{E})^*u\in H^{2*}(M;\Z )$.  

If $\mathscr{V} (M;\theta )\neq \emptyset $, the \emph{K--area of a compact Riemannian manifold $M_g=(M,g)$ w.r.t.~the homology class} $\theta \in H_{2*}(M;\group )$ is defined by
\begin{equation}
\label{defn_k}
\ke (M_g;\theta ):=\left( \inf _{\bundle{E}\in \mathscr{V}(M;\theta )}\| R^\bundle{E}\| _g \right) ^{-1}\in (0,\infty ].
\end{equation}
Moreover, we define for monotonicity reasons $\ke (M_g;\theta )=0$ in case $\mathscr{V} (M;\theta )=\emptyset $ (for instance $\theta =0$). We will frequently use this definition to introduce various K--areas by taking the infimum over different sets of vector bundles. 
\begin{rem}
Because $M$ is compact, any element of the K--group $K(M)$ can be represented by $[\bundle{E}]-[\C ^N]$ where $\bundle{E}$ is a complex vector bundle and $\C ^N$ is a flat bundle on $M$ and moreover, the Chern character map $\mathrm{ch}:K(M)\otimes \Q \to H^{2*}(M;\Q )$ is an isomorphism (cf.~\cite{AH}). Thus,  $\mathscr{V}(M;\theta )$ is nonempty if $\theta $ is a nontrivial element in $H_{2*}(M;\Q )$. 
\end{rem}
Since $\mathscr{V}(M;a\cdot \theta )= \mathscr{V}(M;\theta )$ is independent on the choice of the metric, we conclude the following scaling invariance of the K--area:
\[
\ke (M_{\mu \cdot g};a\cdot \theta )= \mu \cdot \ke (M_g;\theta )
\]
where $\mu $ is a positive constant and $a\in \Z \setminus \{ 0\} $. In order to extend the above definition for odd homology classes, we add large circles. In fact, suppose $ \theta \in H_{2*+1}(M;\group )$ then $\theta \times [S^1 ] \in H_{2*}(M\times S^1;\group )$ for a fundamental class of $S^1$. Thus, we define
\[
\ke (M_g;\theta ):=\sup _{\mathrm{d}t^2}\ke (M_g\times S^1_{\mathrm{d}t^2};\theta \times [ S^1] )
\]
where the supremum runs over all line elements $\mathrm{d}t^2$ of $S^1$ and $M_g\times S^1_{\mathrm{d}t^2}$ is endowed with the product metric. Note that the K--area on the right hand side does not depend on the choice of the fundamental class $[S^1]$. Replacing $\mathrm{d}t^2$ by $\mu \cdot \mathrm{d}t^2=\mathrm{d}\tilde t^2$ we obtain the above scaling invariance for odd homology classes. The K--area of a general class $\theta \in H_*(M;\group )$ is given by
\[
\ke (M_g;\theta ):=\max \{ \ke (M_g;\theta _\mathrm{even}),\ke (M_g;\theta _\mathrm{odd})\} .
\]
In case $\theta =[ M] $ we omit $\theta $ and simply write $\ke (M_g)$ which is the total K--area of $M$ introduced by Gromov in \cite{Gr01} and considered in \cite{Dav,Entov,Polt,pre_Lima,pre_Saval}. 
\begin{prop}
\label{proposition1}
Suppose $\theta =\sum _i\theta _i$ with $\theta _i\in H_{i}(M ^n;\group )$, then
\[
\ke (M_g;\theta )=\max \{ \ke (M_g;\theta _i)\ |\ i=0\ldots n\} .
\] 
Moreover, we obtain for $\theta ,\eta \in H_k(M ;\group )$
\[
\ke (M_g;\theta +\eta )\leq \max \{ \ke (M_g;\theta ),\ke (M_g;\eta )\} .
\]
\end{prop}
\begin{proof}
We start with the case $\theta \in H_{2*}(M;\group )$. Observe that $\mathscr{V}(M;\theta _i)\subset \mathscr{V}(M;\theta )$ because $0\neq \rho ^\bundle{E}_*(\theta _i) \in H_i(\BU ;\group )$ implies $\rho _*^\bundle{E}(\theta )\neq 0$. Conversely, if $\rho _*^\bundle{E}(\theta )\neq 0$, then $\rho ^\bundle{E}_*(\theta _i)\neq 0$ for at least one $\theta _i$ which completes the proof for even homology classes: $\mathscr{V}(M;\theta )=\bigcup _i\mathscr{V}(M;\theta _i)$. For the second statement, we obtain by the same argument $\mathscr{V}(M;\theta +\eta )\subset \mathscr{V}(M;\theta )\cup \mathscr{V}(M;\eta )$ which shows the inequality if $k$ is even. Now suppose that $\theta \in H_{2*+1}(M;\group )$, then:
\[
\begin{split}
\ke (M_g;\theta )&=\sup _{\mathrm{d}t^2}\max \{ \ke (M_g\times S^1_{\mathrm{d}t^2};\theta _i\times [S^1])\ |\ i=1\ldots n\} \\
&=\max \Bigl\{ \sup _{\mathrm{d}t^2}\ke (M_g\times S^1_{\mathrm{d}t^2};\theta _i\times [S^1])\ |\ i=1\ldots n \Bigl\} .
\end{split}
\] 
The general case is an easy consequence. 
\end{proof}
Since $M$ is compact, for any two Riemannian metrics $g$ and $h$ on $M$ there is a constant $C\geq 1$ such that $C^{-1}\cdot \| R^\bundle{E}\| _g\leq \| R^\bundle{E}\| _h\leq C\cdot \| R^\bundle{E}\| _g$ for all bundles $\bundle{E}$. Hence, the condition $\ke (M_g;\theta )<\infty $ does not depend on the choice of the Riemannian metric on $M$ which yields the following: For each $j$, the set
\[
\H _j(M;\group ):=\{ \theta \in H_j(M;\group )\ | \ \ke (M_g;\theta )<\infty \} \subset H_j(M;\group )
\]
is  a subgroup independent on the choice of the metric $g$ and satisfies
\[
\H _*(M;\group )=\{ \theta \in H_*(M ;\group )\ | \ \ke (M_g;\theta )<\infty \} =\bigoplus \H _j(M;\group ).
\]
If $\theta \in H_0(M;\group )$ does not vanish, there are trivial bundles $\bundle{E}$ with $\rho ^\bundle{E}_*(\theta )\neq 0$. In this case we use that the fiber dimension of vector bundles is not assumed to be globally constant. Because trivial bundles admit flat connections, the K--area of $0\neq \theta \in H_0(M;\group )$ is infinite which implies $\H _0(M;\group )=\{ 0\} $. Moreover, if $\group $ is a ring, $\mathscr{V}(M;a\cdot \theta )\subset \mathscr{V}(M;\theta )$ for all $\theta \in H_{2*}(M;G)$ yields
\[
\ke (M_g;a\cdot \eta )\leq \ke (M_g;\eta )
\]
for all $\eta \in H_*(M;\group )$ and $a\in \group $. Since $H_*(\mathrm{BU}_n;\group )$ is a free $\group $--module, this is an equality if $a\neq 0$ and $\group $ has no zero divisors. Hence, for any coefficient ring $\group $, $\H _j(M;\group )$ is a $\group $--submodule of $H _j(M;\group )$. 

Let $M\coprod N$ be the disjoint sum of $M$ and $N$, then for any $\theta \in H_{2*}(M\coprod N;G)$ the interesting bundles in $\mathscr{V}(M\coprod N;\theta )$ are determined by $\mathscr{V}(M;\theta _{|M})\cup \mathscr{V}(N;\theta _{|N})$  where $\theta _M$ and $\theta _N$ mean the restriction of $\theta $ to $M$ and $N$. Thus, the K--area of $\theta \in H_k(M\coprod  N;G)$ equals the maximum of the K--area of $\theta _M$ and the K--area of $\theta _N$ which proves the additivity axiom
\[
\H _k\left( M\coprod N;G\right) \cong \H _k(M;G)\oplus \H _k(N;G).
\]
\begin{prop}
\label{proposition2}
Let $f:(M,g)\to (N,h)$ be a smooth map with $g\geq f^*h$ on $\Lambda ^2TM$, then
\[
\ke (M_g;\theta )\geq \ke (N_h;f_*\theta ).
\]
In fact, for each continuous map $f:M\to N$, the induced homomorphism on singular homology yields homomorphisms $f_*:\H _j(M;\group )\to \H _j(N;\group )$.

Moreover, if $M\stackrel{i}{\hookrightarrow} N$ is a compact submanifold and a retract, then for suitable metrics $h$ on $N$
\[
\ke (M_h;\theta )=\ke (N_h;i_*\theta ),
\]
here suitable means that the smooth retraction map $r:N\to M$ is $1$--Lipschitz, i.e.~$r^*h_{|M}\leq h$ on $\Lambda ^2TN$. 
\end{prop}
\begin{proof}
We start with the case $\theta \in H_{2*}(M;\group )$. Since
\[
\bigl( \rho ^{f^*\bundle{E}}\bigl) _*(\theta )=\left( \rho ^\bundle{E}\circ f\right) _*(\theta )=\rho ^\bundle{E}_*(f_*\theta ),
\]
the pull back of vector bundles yields a map $f^*:\mathscr{V}(N;f_*\theta )\to \mathscr{V}(M;\theta )$. Moreover, $g\geq f^*h$ on $\Lambda ^2TM$ supplies
\[
\| R^{f^*\bundle{E}}\| _g\leq \| R^\bundle{E}\| _h
\]
which proves the inequality. If $\theta \in H_{2*+1}(M;\group )$ we consider the map $f\times \mathrm{id}:M\times S^1\to N\times S^1$ and apply the case for even homology classes. The necessary inequality on $\Lambda ^2T(M\times S^1)$ follows from the compactness of $M$ because for any line element $\mathrm{d}t^2$, there is some $\mathrm{d}\tilde t^2$ such that $g\oplus \mathrm{d}\tilde t^2\geq  (f\times \mathrm{id})^*(h\oplus \mathrm{d}t^2)$ on $\Lambda ^2T(M\times S^1)$. The second observation is proved for smooth $f:M\to N$ by the inequality and in case of continuous $f$ we use the smooth approximation theorem and the homotopy invariance of the induced homomorphism. For the last statement it remains to show ''$\leq $'' but this follows by considering a smooth retraction map $r:N\to M$ (using the smooth approximation theorem: every retract $M\subset N$ for compact manifolds $M$ and $N$ is also a smooth retract). 
\end{proof}

This shows that $\H _*(.;G)$ is a functor on the category of compact smooth manifolds and continuous maps into the category of graded abelian groups which satisfies the homotopy, dimension and additivity axiom. In fact, $\H _*(M;G)$ depends only on the homotopy type of $M$.  Gromov proved in \cite{Gr01} that the total K--area of simply connected manifolds and spin manifolds of positive scalar curvature is finite which means $\H _n(M^n)=H_n(M) $ for these closed manifolds. Furthermore, using the above proposition and the observation that $H_{2j}(\C P^n )$ is generated by the fundamental class of $\C P^j\subset \C P^n$ we conclude for the complex projective spaces
\[
\H _k(\C P^n)=\biggl\{ \begin{array}{cl}
\Z &if \ \ k\in \{ 2,4,\ldots ,2n\}\\
0& otherwise.
\end{array}
\]
Conversely, every closed connected orientable surface $M$ of positive genus has infinite total K--area, i.e.~$\H _2(M)=\{ 0\}$ (cf.~\cite{Gr01}). If $\H _k(N)$ is known, the previous proposition is one of best ways to compute $\H _k(M)$ by considering maps $M\to N$ respectively $N\to M$.  
\begin{cor}
The Hurewicz homomorphism satisfies for all $j\geq 2$:
\[
h_j:\pi _j(M)\to \H _j(M)\subset H_j(M).
\]
\end{cor}
\begin{proof}
Let $f:(S^j,e)\to (M,x)$ be a representative of $\alpha \in \pi _j(M,x)$, then
\[
\infty > \ke (S^j_g;[S^j])\geq \ke (M_{\overline{g}};f_*[S^j])=\ke (M_{\overline{g}};h_j(\alpha ))
\]
for suitable metrics $g$, $\overline{g}$ on $S^j$ and $M$. 
\end{proof}

\begin{prop}
\begin{enumerate}
\item If $T\subset H_*(M)$ is the torsion subgroup, then $T\subset \H _*(M)$.
\item $\H _*(M;\Q )\cong \H _*(M)\otimes \Q $.
\item $\H _1(M;\Q )=\{ 0\} $.
\end{enumerate}
\end{prop}
\begin{proof}
Since $H_*(\mathrm{BU}_n)$ is free, each induced homomorphism $H_*(M)\to H_*(\BU )$ maps $T$ to $0$. Hence, for any $\theta \in T$, $\mathscr{V}(M;\theta )=\emptyset $ and $\mathscr{V}(M\times S^1;\theta \times [S^1])=\emptyset $ supply $\ke (M_g;\theta )=0$, i.e.~$\theta \in \H _*(M)$. For the second claim we consider the commutative diagram
\[
\begin{xy}
\xymatrix{H_{2k}(M)\otimes \Q \ar[d]^{\lambda_M }\ar[r]^{\rho _*^\bundle{E}\otimes \mathrm{id}} &H_{2k}(\BU )\otimes \Q \ar[d]^\lambda \\
H_{2k}(M;\Q )\ar[r]^{\rho _*^\bundle{E}}&H_{2k}(\BU ;\Q )}
\end{xy}
\]
where $\lambda $ and $\lambda _M$ are isomorphism. Hence, $\mathscr{V}(M;\theta )=\mathscr{V}(M;\lambda _M(\theta \otimes x))$ for all $\theta \in H_{2k}(M)$ and $x\in \Q \setminus \{ 0\}$ prove that $\lambda _M:\H _{2k}(M)\otimes \Q \to \H _{2k}(M;\Q )$ is an isomorphism. In order to see $\H _{2k+1}(M)\otimes \Q \cong \H _{2k+1}(M;\Q )$ replace $M$ by $M\times S^1$ in the above diagram. Suppose $\theta \in H_1(M;\Q )$ is nontrivial, then there is some $\alpha $ in the lattice $H^1(M;\Z )\subset  H^1(M;\Q )$ with $\left< \alpha ,\theta \right> \neq 0$. Moreover, by the Hopf theorem and the smooth approximation theorem there is a smooth map $f:M\to S^1$ uniquely determined up to homotopy with $f^*\omega =\alpha $ where $\omega \in H^1(S^1;\Z )$ is the orientation class.  Consider the map $f\times \mathrm{id}:M\times S^1\to T^2$, then
\[
\left< \omega ,f_*\theta \right> =\left< f^*\omega ,\theta\right> =\left< \alpha ,\theta \right> \neq 0
\]
proves $f_*\theta \neq 0$ which yields $[T^2]=x\cdot f_*\theta \times [S^1]$ for some $x\in \Q \setminus \{ 0\} $. Hence, the scaling invariance in the homology class and proposition \ref{proposition2} show
\[
\ke (M_g\times S^1_{ \mathrm{d}t^2};\theta \times [S^1])\geq \ke (T^2_0;[T^2])=\infty .
\]
where $T_0^2$ is a suitable flat torus ($T^2$ has infinite total K--area by the remarks below or use the result in \cite{Gr01}).
\end{proof}
Although the torsion subgroup remains unchanged for $\Z $ coefficients, this does not have to hold for arbitrary coefficient groups. For instance, consider the real projective space $\R P^n$ in dimension $n\geq 2$ and let $\bundle{E}$ be the canonical complex line bundle with $0\neq c_1(\bundle{E})\in H^2(\R P^n;\Z )= \Z _2 $. The image of $c_1(\bundle{E})$ in $H^2(\R P^n;\Z _2)$ is the second Stiefel--Whitney class $w_2(\bundle{E})$ which is nontrivial. Hence, if $\theta $ denotes the generator of $ H_2(\R P^n;\Z _2)$, $\left< w_2(\bundle{E}),\theta \right> \neq 0$ implies $\bundle{E}\in \mathscr{V}(\R P^n;\theta )$. But the real Chern class of $\bundle{E}$ vanishes which means that $\bundle{E}$ admits a flat connection. This shows $\ke (\R P^n;\theta )=\infty $ and $\H _2(\R P^n;\Z _2)=\{ 0\} $ whereas  $H _2(\R P^n;\Z _2)=\Z _2$.   

\section{K--area for the Chern character}
In the above definition of the K--area we considered Hermitian bundles which have a nontrivial $\theta $--Chern number. However, we can also restrict to bundles with a specific nontrivial $\theta $--Chern number respectively fix a characteristic class $u\in H^*(\BU ;\group )$ and consider bundles with $\left< u(\bundle{E}),\theta \right> \neq 0$ where $u(\bundle{E}):=(\rho ^\bundle{E})^*u$. This leads to the $\mathrm{K} _u$--area and results similar to the one presented above with a few exceptions. But we still obtains subgroups $\H _k(M;u)\subset H_k(M;\group )$ of homology classes with finite $\mathrm{K}_u$--area and moreover, a continuous map $f:M\to N$ induces homomorphisms $f_*:\H _k(M;u)\to \H _k(N;u )$. Hence, each characteristic class $u\in H^*(\BU ;G)$ determines a functor $\H _*(.;u)$ on the category of compact smooth manifolds and continuous maps into the category of graded abelian groups which satisfies the homotopy, dimension and additivity axiom. 

Another important K--area is the K--area for the Chern character. This area is particularly interesting for scalar curvature results. If $\theta \in H_{2*}(M;\Q )$, we denote by $\mathscr{V}_\mathrm{ch} (M;\theta )$ the set of Hermitian bundles $\bundle{E}\to M$ endowed with a Hermitian connection such that $\left< \mathrm{ch}(\bundle{E})),\theta \right> \neq 0$ for the Chern character $\mathrm{ch}(\bundle{E})$. As already remarked $\mathscr{V}_\mathrm{ch} (M;\theta )$ is nonempty if $\theta \neq 0$, i.e.~we define $\ke _\mathrm{ch}(M_g;\theta )$ as above but take the infimum in (\ref{defn_k}) over bundles in $\mathscr{V}_\mathrm{ch} (M;\theta )$. We set $\ke _\mathrm{ch}(M_g;0 )=0$ and add large circles to define the $\mathrm{K} _\mathrm{ch}$--area for odd homology classes:
\[
\ke _\mathrm{ch}(M_g;\theta ):=\sup _{\mathrm{d}t^2}\ke _\mathrm{ch}(M_g\times S^1_{\mathrm{d}t^2};\theta \times [S^1]),\qquad \theta \in H_{2*+1}(M;\Q ).
\]
The $\mathrm{K} _\mathrm{ch}$--area of a general class is the maximum of the even and the odd part. Since $\mathscr{V}_\mathrm{ch}(M;\theta )\subset \mathscr{V}(M;\theta )$, we conclude
\[
\ke _\mathrm{ch}(M_g;\theta )\leq \ke (M_g;\theta )
\]
for all $\theta \in H_*(M;\Q )$. Equality seems to hold only in few examples, for instance $\ke _\mathrm{ch}(M_g)=\ke (M_g)$ if $M$ is homeomorphic to $S^n$ or $\ke _\mathrm{ch}(M_g;\theta )=\ke (M_g;\theta )$ if $\theta \in H_2(M;\Q )$.  Thus, if $\H _j(M;\mathrm{ch})$ denotes the set of all homology classes $\theta \in H_j(M;\Q )$ with $\ke _\mathrm{ch}(M_g;\theta )<\infty $, then $\H _j(M;\Q )\subset \H _j(M;\mathrm{ch})$ are linear subspaces of $H_j(M;\Q )$ for each $j$.

\begin{prop}
\label{proposition6}
Suppose $\theta \in H_{2*}(M;\Q )$ is of total degree $2m>0$, then
\[
\ke (M_g;\theta )\leq m^2\cdot \ke _\mathrm{ch}(M_g;\theta ).
\]
In particular, $\H _j(M;\Q )=\H _j(M;\mathrm{ch})$ for each $j$.
\end{prop}
\begin{proof}
The proof follows mainly the idea in \cite{Gr01}, only a couple of changes are necessary to compute the precise estimate. We show that for any bundle $\bundle{E}\in \mathscr{V}(M;\theta )$ there is an associated bundle $\bundle{X}\in \mathscr{V}_\mathrm{ch} (M;\theta )$ in such a way that for all $x,y\in TM$:
\[
| R^\bundle{X}(x\wedge y )| _{op}\leq m^2\cdot | R^\bundle{E}(x\wedge y )| _{op}.
\]
Then the above inequality is an easy consequence. We assume without loss of generality $\theta _0=0$, because otherwise the K--and $\mathrm{K}_\mathrm{ch}$--area are infinite and the statement is trivial. Essential to us will be the proposition that the Adams operation $\psi _k$  applied to a vector bundle is nothing but a polynomial in the exterior powers of this bundle. In fact, $\psi _k(\bundle{E})\in K(M) $ is a $\Z $--linear combination of
\[
\Lambda ^\alpha \bundle{E}=\Lambda ^{\alpha _1}\bundle{E}\otimes \cdots \otimes \Lambda ^{\alpha _l}\bundle{E}
\]
where $\alpha \in \N ^l_0$ is a multi index with $|\alpha |=k$. The $\Z $--coefficient for $\alpha $ in this linear combination is simply determined by the corresponding coefficient in the representation of the $k$th power sum by elementary symmetric polynomials.   

\emph{Claim 1:} Suppose $\bundle{E}$ is a vector bundle with $\left< \mathrm{ch}(\Lambda ^\alpha \bundle{E}),\theta \right> =0$ for all multi indexes $\alpha $ with $1\leq |\alpha |\leq m$, then
\[
\left< \mathrm{ch}_i(\bundle{E}),\theta _{2i}\right> =0
\]
holds for all $i$. Since $\theta _{2i}=0$ for $i>m$, it suffices to prove the case $i\leq m$. This can be seen by using the Adams operation $\psi _k$ for $k\in \{ 1,\ldots ,m\}$. Under the assumption on $\bundle{E}$, $\left< \mathrm{ch}(\psi _k\bundle{E}),\theta \right> $ vanishes for all $k\in \{ 1,\ldots ,m\} $. Since $\mathrm{ch}(\psi _k\bundle{E})=\rho _k(\mathrm{ch}(\bundle{E}))$ with $\rho _k=k^j$ on $\mathrm{H}^{2j}(M,\Q )$, we obtain
\[
0=\left< \mathrm{ch}(\psi _k\bundle{E}),\theta \right> =\sum _{i=1}^mk^i\left< \mathrm{ch}_i(\bundle{E}),\theta _{2i}\right> =\left< a_k,b\right> _{Euc}
\]
for all $k\in \{ 1,\ldots ,m\} $ where
\[
a_k=(k^1,k^2,\ldots ,k^m)\quad \text{and}\quad b=\bigl(\left< \mathrm{ch}_1(\bundle{E}),\theta _2\right> ,\ldots ,\left< \mathrm{ch}_m(\bundle{E}),\theta _{2m}\right> \bigl) .
\]
However, the vectors $a_ 1,\ldots ,a_m$ form a basis of $\Q ^m$ which implies that $b=0$.

\emph{Claim 2:} Let $\bundle{E}$ be a vector bundle and fix some $i\in \{ 1,\ldots ,m\} $. If $\left< \mathrm{ch}_i(\Lambda ^\alpha \bundle{E}),\theta _{2i}\right> $ vanishes for all multi indexes $\alpha =(k_1,\ldots ,k_l)\in \N ^l$ with $|\alpha |=\sum k_j=i$, then for each polynomial $p$: $\left< p(c(\bundle{E})),\theta _{2i}\right> =0$. We use again the Adams operation in order to see this. Define
\[
B_i^\alpha :=\left< \mathrm{ch}_i(\psi _\alpha (\bundle{E})),\theta _{2i}\right> =\left< \mathrm{ch}_i(\psi _{k_1}(\bundle{E})\otimes \cdots \otimes \psi _{k_l}(\bundle{E})),\theta _{2i}\right>  
\]
for multi indexes $\alpha =(k_1,\ldots ,k_l)\in \N ^l$ with $|\alpha |=\sum k_j=i$. Since $\psi _{\alpha }(\bundle{E})$ is a linear combination of $\Lambda ^\beta \bundle{E}$ with $|\beta |\leq i$, $B^\alpha _i$ vanishes for all $\alpha $ under the above assumption on $\bundle{E}$. We will show that $B^\alpha _i=0$ for all $\alpha $ with $|\alpha |=i$ implies
\[
\left< \prod _{j=1}^i\mathrm{ch}_j(\bundle{E})^{\beta _j},\theta _{2i}\right> =0
\]
for all $\beta _1,\ldots ,\beta _i\in \N _0$ with $i=\sum j\cdot \beta _j$. Since the Chern classes of degree $\leq j$ are polynomials in the Chern characters $\mathrm{ch}_1,\ldots ,\mathrm{ch}_j$, this concludes for all polynomials $p$: $\left< p(c(\bundle{E})),\theta _{2i}\right> =0$. Consider a formal power series in $t$: $a(t)=\sum _ja_j\cdot t^j$ with $a_0\neq 0$ and let $\alpha =(k_1,\ldots ,k_i)\in (\N _0)^i$, then
\[
A^\alpha (t):= \prod _{j=1}^i a(k_j\cdot t)
\]
is symmetric in $k_j$ and we obtain an expansion $A^\alpha (t)=\sum A^\alpha _it^i$ where $A_i^\alpha $ does not depend on $t$ and is a polynomial in $a_1,\ldots ,a_i$. A little exercise in linear algebra shows
\[
\Q [a_1,\ldots ,a_i]_{i}=\{ a\in \Q [a_1,\ldots ,a_i]\ |\ \deg a=i\} =\mathrm{span}_\Q  \{ A^\alpha _i\ |\  |\alpha |=i\} 
\]
where $\deg a_1^{j_1}\cdots a_i^{j_i}=\sum \limits_{s=1}^is\cdot j_s$. Note that the inclusion ''$\supseteq $'' follows by definition and ''$\subseteq $'' follows from dimension reasons because if $\{ \alpha _s\in (\N _0)^i|s\in S\}$ is a set of multi indexes such that for all $r\neq s$ there is at least one elementary symmetric polynomial $\sigma _l$ with $\sigma _l(\alpha _r)\neq \sigma _l(\alpha _s)$, then $\{ A^{\alpha _s}_i\ |\ s\in S\} $ is linear independent in $\Q [a_1,\ldots ,a_i]_i$. If we identify $a_j$ with $\mathrm{ch}_j(\bundle{E}) $ for all $j$, the properties of the Adams operation and the Chern character imply that $A^\alpha _i$ is determined by
\[
\mathrm{ch}_i(\psi _{k_1}(\bundle{E})\otimes \cdots \otimes \psi _{k_l}(\bundle{E}))=\mathrm{ch}_i(\psi _\alpha (\bundle{E})).
\]
Thus, $B^\alpha _i=\left< A_i^\alpha ,\theta _{2i}\right> =0$ for all $\alpha $ proves $\left< b,\theta _{2i}\right> =0$ for all $b\in \Q [a_1,\ldots ,a_i]_i$.

\emph{Conclusion:} Suppose $\bundle{E}\in \mathscr{V}(M;\theta )$ and consider the bundle $\bundle{X}:=\Lambda ^\alpha \left( \Lambda ^{\beta }\bundle{E}\right) $ for multi indexes $\alpha $, $\beta $ with $|\alpha |, |\beta |\leq m$, then $\| R^\bundle{X}\| _g\leq m^2\| R^\bundle{E}\| _g$ is satisfied. Moreover, we conclude $\bundle{X}\in \mathscr{V}_\mathrm{ch}(M;\theta )$ for some $\alpha $, $\beta $ by the following contradiction argument. If $\left< \mathrm{ch}(\bundle{X}),\theta \right> $ vanishes for all $\alpha $ and $\beta $, we apply claim 1 to the bundle $\bundle{Y}=\Lambda ^\beta \bundle{E}$ and conclude $\left< \mathrm{ch}_i(\Lambda ^\beta \bundle{E}),\theta _{2i}\right> =0$ for all $i$ and all $\beta $. Hence, claim 2 proves that for any polynomial $p$: $\left< p(c(\bundle{E})),\theta \right> =0$, i.e.~$\bundle{E}\notin \mathscr{V}(M;\theta )$ a contradiction to our assumption. 
\end{proof}

\begin{prop}
\label{proposition7}
Let $f:(\tilde M,\tilde g)\to (M,g)$ be a Riemannian covering with $M$, $\tilde M$ closed and oriented, then
\[
\ke _\mathrm{ch}(\tilde M_{\tilde g};f^!\theta )=\ke _\mathrm{ch}(M_g;\theta )
\]
where $f^!:H_*(M;\Q )\to H_*(\tilde M;\Q )$ is the transfer homomorphism. 
\end{prop}
\begin{proof}
Observe that $f$ is a finite covering and $0<|\deg f|$ is the number of sheets. Since $f_*f^!\theta =\deg f\cdot \theta $ we conclude ''$\geq $'' from the $\mathrm{K}_\mathrm{ch}$--version of proposition \ref{proposition2} (pull back of vector bundles). Furthermore, each bundle $\bundle{E}\to \tilde M$ induces a bundle $\bundle{E}'\to M$ by
\[
\bundle{E}'_x=\bigoplus _{y\in f^{-1}(x)}\bundle{E}_y.
\]
The connection on $\bundle{E}$ yields a connection on $\bundle{E}'$ with $\| R^{\bundle{E}'}\| _g=\| R^\bundle{E}\| _{\tilde g}$. Thus, we conclude ''$\leq $'' for $0\neq \theta \in H_{2*}(M;\Q )$ if $\bundle{E}\in \mathscr{V}_\mathrm{ch}(\tilde M;f^!\theta )$ implies $\bundle{E}'\in \mathscr{V}_\mathrm{ch}(M;\theta )$ (note that $f^!$ is injective). Assume that $f$ is a normal covering, then $f^*\bundle{E}'$ is isomorphic to $\bigoplus h^*\bundle{E}$ where the sum is taken over all deck transformations $h$. Hence, we obtain
\[
\begin{split}
\deg f\cdot \left< \mathrm{ch}(\bundle{E}'),\theta \right> &=\left< \mathrm{ch}(\bundle{E}'),f_*f^!\theta \right> =\left< \mathrm{ch}(f^*\bundle{E}'),f^!\theta \right> \\
&=\sum _h\left< h^*\mathrm{ch}( \bundle{E})  ,f^!\theta \right> =|\deg f|\cdot \left< \mathrm{ch}(\bundle{E}),f^!\theta \right> 
\end{split}
\]
because for a deck transformation $h:\tilde M\to \tilde M$, $h_*$ acts as $\mathrm{Id}$ on the image of $f^!$: $f=f\circ h$ and $h_*h^!=h^!h_*=\mathrm{Id}$ yield $h_*f^!=h_*(fh)^!=h_*(h^!f^!)=f^!$. This proves the claim if $\theta \in H_{2*}(M;\Q )$ and $f$ is a normal covering. If $f$ is not a normal covering, let $H\subset f_\# \pi _1(\tilde M,\tilde x)$ be a normal subgroup of $\pi _1(M,f(\tilde x))$ such that $\pi _1(M,f(\tilde x))/H$ is a finite group [$H$ always exists because $f_\# \pi _1(\tilde M,\tilde x)$ has finite index in $\pi _1(M,f(\tilde x))$]. Hence, the associated (smooth) covering $l:(N,y)\to (M,f(\tilde x))$ with $l_\# \pi _1(N,y)=H$ is normal and finite, and there is a unique (smooth) map $p:(N,y)\to (\tilde M,\tilde x)$ which is a normal covering and satisfies $l=f\circ p$. Thus, the previous case proves for $\bar g:=l^*g$ and $\theta \in H_{2*}(M;\Q )$:
\[
\ke _\mathrm{ch}(M_g;\theta )=\ke _\mathrm{ch}(N_{\bar g};l^!\theta )=\ke _\mathrm{ch}(N_{\bar g};p^!f^!\theta )=\ke _\mathrm{ch}(\tilde M_{\tilde g};f^!\theta ).
\]
In order to show the proposition for $\theta \in H_{2*+1}(M;\Q )$ consider the covering $b:=f\times \mathrm{id}:\tilde M\times S^1\to M\times S^1$, use $b^!(\theta \times [S^1])=(f^!\theta )\times [S^1]$ and apply the case for even classes to the covering $b$ and the class $\theta \times [S^1]\in H_{2*}(M;\Q )$.
\end{proof}
This proposition allows us to compute the K--areas for the torus respectively for the connected sum of ''nice'' manifolds with a torus. Consider the $2^n$--fold Riemannian covering $f:(T^n,4g)\to (T^n,g),\ \alpha \mapsto 2\alpha $, then by the scaling property
\[
\ke _\mathrm{ch}(T^n_g;\theta )=\ke _\mathrm{ch}(T^n_{4g};f^!\theta )=4\cdot \ke _\mathrm{ch}(T^n_g;\theta )
\]
which shows $ \ke (T^n_g;\theta )=\ke _\mathrm{ch}(T^n_g;\theta )=\infty $ for each $\theta \neq 0$. Here we use that the $\mathrm{K}_\mathrm{ch}$--area of a nontrivial class $\theta \in H_*(M;\Q )$ is in $(0,\infty ]$. This and the following remark prove that $\H _*(T^n\# \cdots \# T^n)=\{ 0\} $ for any finite number of tori. 
\begin{rem}
\label{zusammen}
Suppose that $M^n$ and $N^n$ are closed, oriented and connected. Consider the projection maps $M\# N\to M$ and $M\# N \to N$ by identifying the summand $M$ respectively $N$ to a point, then $\H _k(M\# N)\to \H _k(M)\oplus \H _k(N)$ is well defined and injective, i.e.~$\H _k(M\# N) $ can be considered as a subgroup of $\H _k(M)\oplus \H _k(N)$ for all $k<n$. Moreover, $\H _n(M\# N)=\Z $ implies $\H _n(M)=\Z $ and $\H _n(N)=\Z $. Conversely, if $X\subset M$ or $Y\subset N$ are compact submanifolds of codimension at least one which give a finite upper bound for K--areas on $M$ respectively $N$, then $X,Y\subset M\# N$ determine upper bounds for the corresponding K--areas on $M\# N$. For instance, if $0<j<n$, $\C P^j\subset \C P^n\# T^{2n}$ supplies finite K--area for the homology class induced by the fundamental class of $\C P^j$ which means    
\[
\H _k(\C P^n\# T^{2n})=\biggl\{ \begin{array}{cl}
\Z &if \ \ k\in \{ 2,4,\ldots ,2n-2\}\\
0& otherwise.
\end{array}
\]
\end{rem}
\begin{rem}
Consider a simply connected and closed manifold $N$, then there is some $\epsilon >0$ such that any $\epsilon $--flat bundle on $N$ is already trivial (cf.~Gromov \cite{Gr01}). Hence, the K--area of a class $\theta \in H_{2k}(N)$ is finite if $k>0$ which means $\H _{2k}(N)=H_{2k}(N)$ for all $k>0$. Moreover, the last proposition and the above observations show that $\H _{2k}(M)=H _{2k}(M)$ for all $k>0$ if $M$ is a closed orientable manifold with finite fundamental group. The equality might also hold for odd homology classes, but in this case a more detailed estimate is necessary.  
\end{rem}
\section{Stabalizing K--area and products}
The K--area has a further generalization which depends on taking products with standard tori $T^i=S^1\times \cdots \times S^1$. Suppose that $\theta \in H_{2*-i}(M;\Q )$ for $i\geq 0$, then the $\mathrm{K}_\mathrm{ch}^i$--area of $\theta $ is defined by
\[
\ke _\mathrm{ch}^i(M_g;\theta ):=\ke _\mathrm{ch}(M_g\times T^i_h;\theta \times [T^i])
\]
where the right hand side does not depend on the choice of the metric $h$ by proposition \ref{proposition7}. Note that $\ke ^0_\mathrm{ch}(M_g;\theta )$ is the above $\mathrm{K}_\mathrm{ch}$--area for $\theta \in H_{2*}(M;\Q )$ and  $\ke ^1_\mathrm{ch}(M_g;\theta )$ coincides with the above definition of the $\mathrm{K}_\mathrm{ch}$--area for $\theta \in H_{2*+1}(M;\Q )$. The proposition below shows that $\ke _\mathrm{ch}^i(M;\theta )$ is nondecreasing in $i$ which means $\ke _\mathrm{ch}^{i+2}(M_g;\theta )\geq \ke ^i_\mathrm{ch}(M_g;\theta )$ for all $\theta \in H_{2*-i}(M;\Q )$. The more challenging question is to compute estimates in the opposite direction. We also consider the corresponding stabalized version of this area:
\[
\ke _\mathrm{ch}^\mathrm{st}(M_g;\theta ):=\sup _i\ke _\mathrm{ch}^i(M_g;\theta )\geq \ke _\mathrm{ch}(M_g;\theta ).
\]
We could introduce similar objects for the ordinary K--area, however, it is uncertain if a stabalized version of the K--area can be estimated by the stabalized $\mathrm{K} _\mathrm{ch}$--area (compare the constant in proposition \ref{proposition6}). Since the stabalized $\mathrm{K}_\mathrm{ch}$--area inherits most of the properties of the $\mathrm{K}_\mathrm{ch}$--area, the set $\H _k^\mathrm{st}(M;\Q )\subset H _k(M;\Q )$ consisting of rational homology classes with finite stabalized $\mathrm{K}_\mathrm{ch}$--area is a linear subspace of $\H _k(M;\Q )$ for all $k$. Moreover, a continuos map $f:M\to N$ yields homomorphisms $f_*:\H _k^\mathrm{st}(M;\Q )\to \H _k^\mathrm{st}(N;\Q )$. Hence, the stabalized K--area homology $\H ^\mathrm{st}_k(M;\Q )$ depends only on the homotopy type of $M$. There are plenty of manifolds $M$ with $\H _*(M;\Q )=\H _*^\mathrm{st}(M;\Q )$, however it seems a rather difficult question if this is true for general $M$. The reason for introducing the stabalized version is the more consistend behaviour in the product case:  
\begin{prop}
Suppose that $\eta \in H_k(M;\Q )$ and $\theta \in H_l(N;\Q )$, then
\[
\begin{split}
\ke _\mathrm{ch}^{i+j}(M_g\times N_h;\eta \times \theta )&\geq \min \left\{ \ke _\mathrm{ch}^i(M_g;\eta ),\ke _\mathrm{ch}^j(N_h; \theta )\right\}\\
\ke _\mathrm{ch}^{\mathrm{st}}(M_g\times N_h;\eta \times \theta )&\geq \min \left\{ \ke _\mathrm{ch}^\mathrm{st}(M_g;\eta ),\ke _\mathrm{ch}^\mathrm{st}(N_h; \theta )\right\}
\end{split}
\]
where $i+k$ and $j+l$ are assumed to be even integers. 
\end{prop}
\begin{proof}
The inequality is obvious if one of the homology classes vanishes, hence $\eta \neq 0$ and $\theta \neq 0$. We set $\tilde \eta :=\eta \times [T^i]$ and $\tilde \theta :=\theta \times [T^j]$ for notational simplicity. Suppose that  $\bundle{E}\in \mathscr{V}_\mathrm{ch}(M\times T^i;\tilde \eta )$ and $\bundle{F}\in \mathscr{V}_\mathrm{ch}(N\times T^j;\tilde \theta )$, then
\[
0\neq \bigl< \mathrm{ch}(\bundle{E}),\tilde \eta \bigl> \cdot \bigl< \mathrm{ch}(\bundle{F}),\tilde \theta \bigl> =\bigl< \mathrm{ch}(\bundle{E})\times \mathrm{ch}(\bundle{F}),\tilde \eta \times \tilde \theta \bigl>  =\bigl< \mathrm{ch}(\pi _M^*\bundle{E}\otimes \pi _N^*\bundle{F}),\tilde \eta \times \tilde \theta \bigl>  
\]
shows $\pi _M^*\bundle{E}\otimes \pi _N^*\bundle{F}\in \mathscr{V}_\mathrm{ch}(M\times T^i\times N\times T^j;\tilde \eta \times \tilde \theta )$. Moreover, we obtain for the tensor product connection on $\bundle{G}:=\pi _M^*\bundle{E}\otimes \pi _N^*\bundle{F}$:
\[
\| R^\bundle{G}\| _{g\oplus h}=\max \left\{ \| R^\bundle{E}\| _g,\| R^\bundle{F}\| _h\right\} .
\]
Hence, considering the pull back of $\bundle{G}$ by the coordinate transposition map $M\times N\times T^{i+j}\to M\times T^i\times N\times T^j$ proves the inequality.
\end{proof}
This proposition yields additional examples for finite K--area. For instance, if $M^m$ and $N^n$ are connected, closed and spin and $M\times N$ admits a metric of positive scalar curvature, then $\H _m(M)=\Z $ or $\H _n(N)=\Z $. Note that in this case neither $M$ nor $N$ need to carry a metric of positive scalar curvature. In general it seems rather difficult to see whether $\theta \times \eta $ has finite K--area even if both $\theta $ and $\eta $ have finite K--area. We observe equality in the previous proposition for $\eta \in H_0(M;\Q )$ or $\theta \in H_0(N;\Q )$ by considering the inclusion maps  (cf.~proposition \ref{proposition2}). Moreover, the following proposition proves finite K--area for homology classes $[M]\times \theta $ if $M$ is a closed spin manifold of positive scalar curvature. 
\begin{prop}
Let $(M^n,g)$ be a connected closed spin manifold of positive scalar curvature and $\theta \in H_l(N;\Q )$ where $N$ is a compact manifold, then 
\[
\ke _\mathrm{ch}^\mathrm{st}\bigl( M_g\times N_h;\bigl(\widehat{\mathbf{A}}(M)\cap [M]\bigl)\times \theta \bigl)\leq \frac{2n(n-(-1)^{n})}{\min \mathrm{scal}_g}.
\]
\end{prop}
\begin{proof}
It is sufficient to prove the proposition for the $\mathrm{K}_\mathrm{ch}$--area, the stabalized version follows from this case by replacing $(N,\theta )$ with $(N\times T^i,\theta \times [T^i])$. The inequality is obvious for $\theta =0$, hence suppose $\theta \neq 0$. We start with the case that $n$ is even. Without loss of generality $l$ is even, otherwise go over to $N\times S^1$ and $\theta \times [S^1]$. A bundle $\bundle{E}\in \mathscr{V}_\mathrm{ch}\bigl( M\times N;\bigl(\widehat{\mathbf{A}}(M)\cap [M]\bigl)\times \theta \bigl) $ yields a family of smooth vector bundles over $M$ parametrized by $N$. Thus, twisting the (constant family of) complex spinor bundle $\spinor M$ with the family of bundles $\bundle{E}$ yields a family of Dirac operators $\dirac ^+:\Gamma (\spinor ^+M\otimes \bundle{E})\to \Gamma (\spinor ^-M\otimes \bundle{E})$ over $M$ parametrized by $N$ where the corresponding index satisfies
\[
\left< \mathrm{ind}\ \dirac ^+,\theta \right> =\left< \mathrm{ch}(\bundle{E}),\left(\widehat{\mathbf{A}}(M)\cap [M]\right) \times \theta \right> \neq 0
\]
Hence, there is a point $x\in N$ such that the associated twisted Dirac operator $\dirac _x:\Gamma (\spinor M\otimes \bundle{E}_x)\to \Gamma (\spinor M\otimes \bundle{E}_x)$ has nontrivial kernel where $\bundle{E}_x$ is the induced bundle over $M\times \{ x\}\subset M\times N$. The integrated version of the Bochner Lichnerowicz formula
\[
\dirac ^2_x=\nabla ^*\nabla +\frac{\mathrm{scal}_g}{4}+\frak{R}^{\bundle{E}_x}
\]
shows that there is a point on $M$ where the minimal eigenvalue of $\frak{R}^{\bundle{E}_x}$ is less or equal to $ -\mathrm{scal}_g/4$. However, the maximal absolute eigenvalue of $\frak{R}^{\bundle{E}_x}$ is bounded by $\frac{n(n-1)}{2}\| R^{\bundle{E}_x}\| $ which implies
\[
\| R^\bundle{E}\| \geq \| R^{\bundle{E}_x}\| \geq \frac{\min  \mathrm{scal}_g}{2n(n-1)}
\]
for all $\bundle{E}\in \mathscr{V}_\mathrm{ch}\bigl( M\times N;\bigl( \widehat{\mathbf{A}}(M)\cap [M]\bigl) \times \theta \bigl) $. If $M$ is odd dimensional, apply the even dimensional case to the manifold $M_g\times S^1_{\mathrm{d}t^2}$.
\end{proof}
\begin{rem}
Equality in the above situation can only hold if $g$ has constant scalar curvature. As an example, we obtain equality for even dimensional spheres of constant curvature. In fact, if $S^{2n}_0$ denotes the standard sphere of constant sectional curvature one, then the complex spinor bundle of $S^{2n}$ satisfies $\| R^{{\scriptsize \spinor }^+}\| =\frac{1}{2}$ and $c_n(\spinor ^+)\neq 0$ which proves $\ke (S^{2n}_0)=\ke _\mathrm{ch}(S^{2n}_0)=\ke _\mathrm{ch}^\mathrm{st}(S^{2n}_0)=2$. 
\end{rem}
\begin{rem}
\label{rem8}If $M$ is a compact manifold, the above proposition shows that $\theta \times [S^n]$ has finite K--area for all $\theta \in H_k(M)$ and $n\geq 2$. Hence, we obtain from singular theory that
\[
\H _{k}(M)\oplus H _{k-n}(M)\to \H _{k}(M\times S^n), \ (\eta ,\theta )\mapsto \eta \times \mathbbm{1}+\theta \times [S^n]
\]
is an isomorphism for all $k$ and $n\geq 2$. The same statement holds of course for the stabalized K--area homology. 
\end{rem}
\begin{prop}
Suppose that $\theta \in H_k(M;\Q )$, $\alpha \in H^1(M;\Q )$ and $i+k+1\in 2\Z $, then:
\[
\ke _\mathrm{ch}^i(M_g;\alpha \cap \theta )\leq \ke _\mathrm{ch}^{i+1}(M_g;\theta ).
\]
\end{prop}
\begin{proof}
Consider a bundle $\bundle{E}\in \mathscr{V}_\mathrm{ch}(M\times T^i;(\alpha \cap \theta )\times [T^i])$ and a smooth map $f:M\to S^1$ with $f^*([S^1]^*) =\alpha $ for the orientation class $[S^1]^* \in H^1(S^1;\Q )$. Denote by $\pi :M\times T^{i}\times S^1\to M\times T^i$ the projection and by $\sigma :M\times T^i\times S^1\to S^1\times S^1, (x,y,t)\mapsto (f(x),t)$. For any $\epsilon >0$ there is a bundle $\bundle{F}\in \mathscr{V}_\mathrm{ch}(S^1\times S^1;[S^1]\times [S^1])$ with $\| R^\bundle{F}\| \leq \epsilon $. Hence, $\bundle{G}=\pi ^*\bundle{E}\otimes \sigma ^*\bundle{F}$ satisfies $\| R^\bundle{G}\| \leq \| R^\bundle{E}\| +\epsilon $ and since $\mathrm{ch}(\bundle{F})=\mathrm{rk}(\bundle{F})+x\cdot ([S^1]^* \times [S^1]^*)$ for some $x\in \Q \setminus \{ 0\} $ we conclude
\[
\begin{split}
\left< \mathrm{ch}(\bundle{G}),\theta \times [T^{i+1}]\right> &=x\cdot \left< \pi ^*\mathrm{ch}(\bundle{E})\cdot \pi ^*(\alpha \times \mathbbm{1})\cdot \sigma ^*(\mathbbm{1}\times [S^1]^*),\theta \times [T^{i}]\times [S^1]\right> \\
&=x\cdot \left< \mathrm{ch}(\bundle{E}),(\alpha \cap \theta )\times [T^i]\right> \neq 0 
\end{split}
\]
which proves the inequality.
\end{proof}
We obtain as a corollary that capping with cohomology classes of degree one yields maps:
\[
\begin{split}
&H^1(M;\Z )\times \H _{2k+1}(M)\to \H _{2k}(M),\ (\alpha ,\theta )\mapsto \alpha \cap \theta \\
&H^1(M;\Q )\times \H_k^\mathrm{st}(M;\Q )\to \H _{k-1}^\mathrm{st}(M;\Q ),\ (\alpha ,\theta )\mapsto \alpha \cap \theta .
\end{split}
\]
In particular, a connected closed manifold $M^n$  with $\H ^\mathrm{st}_n(M;\Q )=\Q $ satisfies $\H ^\mathrm{st}_{n-1}(M;\Q )=H_{n-1}(M;\Q )$ and therefore $\H _{n-1}(M)=H_{n-1}(M)$. This may be seen as the K--area analog of Schoen and Yau's result about positive scalar curvature on minimal hypersurfaces (cf.~\cite{SchY6}). The corresponding statement for manifolds of positive scalar curvature was used by Schick in \cite{Schick1} to give a counterexample to the unstable Gromov--Lawson--Rosenberg conjecture.
\section{K--area on noncompact manifolds}
A noncompact manifold has no fundamental class and in view of the relative index theorem we are rather interested on the K--area defined with bundles of compact support. Hence, in order to get a theory of finite K--area on noncompact manifolds we can consider subgroups of the Borel--Moore homology or subgroups of singular cohomology. Moreover, in order to simplify statements about positive scalar curvature and the $\widehat{A}$--class we choose the cohomology approach. In this section all manifolds are assumed to be connected, oriented and without boundary. Remember that singular homology is dual isomorphic to cohomology with compact support and $H^k(M;\Q )$ is the dual space of $H_k(M;\Q )$. Thus, the bilinear pairing
\[
H^{m-k}(M;\Q )\times H^{k}_\mathrm{cpt}(M;\Q )\to \Q \ ,(\alpha ,\beta )\mapsto \int  _M\alpha \cup \beta 
\]
is well defined and nondegenerate where $m=\dim M$. Moreover, if $f:M^m\to N^n$ is a proper map, the pull back $f^*:H_\mathrm{cpt}^k(N;\Q )\to H_\mathrm{cpt}^k(M;\Q )$ induces the transfer homomorphism
\[
f_!:H^{m-k}(M;\Q )\to H^{n-k}(N;\Q ),\quad \int _N (f_!\alpha )\cup \gamma :=\int _M\alpha \cup f^*\gamma .
\]
If $\alpha \in H^{m-2k}(M^m;\Q )$ is a cohomology class, $\mathscr{V} _{\mathrm{cpt}}(M;\alpha )$ denotes the set of Hermitian vector bundles $\bundle{E}\to M$ endowed with a Hermitian connection such that $\bundle{E}$ and its connection are trivial outside of a compact set and such that there are nonnegative integers $i_1,\ldots ,i_k$ with $\sum j\cdot i_j=k$ and
\[
0\neq \int _M\alpha \cdot c_1(\bundle{E})^{i_1}\cdots c_k(\bundle{E})^{i_k} .
\]
Note that the right hand side is well defined because $c_i(\bundle{E})$ has compact support. The K--area $\ke (M_g;\alpha )$ of the Riemannian manifold $(M,g)$ w.r.t.~the cohomology class $\alpha $ is defined like in (\ref{defn_k}) by taking the infimum over bundles in $\mathscr{V} _\mathrm{cpt}(M;\alpha )$. We set $\ke  (M_g;\alpha )=0$ in case of $\mathscr{V} _\mathrm{cpt}(M;\alpha )=\emptyset $. If $k$ is odd and $\alpha \in H^{m-k}(M^m;\Q )$, define
\[
\ke (M_g;\alpha ):=\sup _{\mathrm{d}t^2}\ke (M_g\times S^1_{\mathrm{d}t^2};\pi ^*\alpha )
\]     
for the projection map $\pi :M\times S^1\to M$. In accordance with proposition \ref{proposition1} the K--area of a general cohomology class $\alpha =\sum \alpha _i$ is the maximum of the K--areas of the $\alpha _i\in H^i(M;\Q )$. If $f:(M^m,g)\to (N^n,h)$ is a smooth proper Lipschitz map in the sense that $g\geq f^*h$ on $TM$, then
\[
\ke (M_g;\alpha )\geq \ke (N_h;f_!\alpha )
\]
for all $\alpha \in H^{*}(M;\Q )$. If $\alpha \in H^{m-2*}(M;\Q )$, this inequality is still true for $\Lambda ^2$--Lipschitz maps $f$ (i.e.~$g\geq f^*h$ on $\Lambda ^2TM$). However, it fails for general $\alpha $ because the condition $\Lambda ^2$--Lipschitz does not extend to the map $f\times \mathrm{id}:M\times S^1\to N\times S^1$ if $M$ is not compact (cf.~proof of proposition \ref{proposition2}). %It suffices to show the inequality for $\alpha \in H^{m-k}(M;\Q )$. If $k$ is even, the pull back of bundles by $f$ yields a map $f^*:\mathscr{V}_\mathrm{cpt}(N;f_!\alpha )\to \mathscr{V}_\mathrm{cpt}(M;\alpha )$, hence $\| R^{f^*\bundle{E}}\| _g\leq \| R^\bundle{E}\| _h$ shows the claim. If $k$ is odd, consider the cohomology class $\pi ^*\alpha $ on $M\times S^1$ and use the same argument on the map $f\times \mathrm{id}:M\times S^1\to N\times S^1$.
Because proposition \ref{proposition1} holds for the cohomology K--area,
\[
\H ^j(M_g;\Q ):=\{ \alpha \in H^j(M;\Q )\ |\ \ke (M_g;\alpha )<\infty \}
\] 
is a linear subspace for each $j$ which of course depends on the asymptotic geometry of the metric $g$ on noncompact manifolds except for $j=m=\dim M$ since flat connections on trivial bundles mean $\H ^m(M;\Q )=\{ 0\}$. The subscript on $M$ refers to the choice of a geometry at infinity. However, if $g\sim h$ in the sense that there is a constant $C>0$ with $C^{-1}\cdot g\leq h\leq C\cdot g$ on $TM$, then $\H ^j(M_g;\Q )=\H ^j(M_h;\Q )$. In fact, on compact manifolds the subspaces $\H ^j(M;\Q )$ are independent on the choice of the metric. Moreover, if $f:(M^m,g)\to (N^n,h)$ is a smooth proper Lipschitz map, the transfer homomorphisms restrict to  homomorphisms on the respective cohomology groups with finite K--area:
\[
f_!:\H ^{m-k}(M_g;\Q )\to \H ^{n-k}(N_h;\Q ).
\]
If $M^m$ is a closed oriented manifold, the above definitions coincide up to Poincare duality, i.e.
\[
\ke (M_g;\alpha )=\ke (M_g;\alpha \cap [M])
\]
where $[M]\in H_{m}(M)$ is the fundamental class. In particular, Poincare duality yields isomorphisms
\[
\cap [M]:\H ^j(M;\Q )\to \H _{m-j}(M;\Q ).
\]
At this point we should clarify that in general $\H ^j(M;\Q )$ is no longer the dual space of $\H _j(M;\Q )$. Although, the cohomology approach has its advantages for noncompact manifolds and statements about scalar curvature, the theory of finite K--area is a homology theory, because the transfer homomorphisms are not natural. 
\begin{prop}
If $M$ is a compactly $\widehat{A}$--enlargeable manifold, then the total $\widehat{A}$--class of $M$ has infinite K--area: $\widehat{\mathbf{A}}(M)\notin \H ^*(M;\Q )$. In particular, a compactly enlargeable manifold $M^m$ satisfies
\[
\H ^0(M;\Q )=\H _m(M;\Q )=\{ 0\} .
\]
Moreover, if $M$ is weakly enlargeable, then for each metric $g$ on $M$ and every constant $C>0$ there is a Riemannian covering $\tilde{M}_{\tilde g}$ with $\ke (\tilde M_{\tilde g};\widehat{\mathbf{A}}(\tilde{M}))>C$. 
\end{prop}
\begin{proof}
Consider the case of a compactly $\widehat{A}$--enlargeable manifold $M_g$, the enlargeable case follows in the same way by replacing the $\widehat{A}$--class with $\mathbbm{1}\in H^0(M;\Q )$ and the $\widehat{A}$--degree with the ordinary degree of a map. $M$ is closed and for each $\epsilon >0$ there is a finite Riemannian covering $(\tilde{M},\tilde{g})\to (M,g)$ and a map $f:(\tilde M,\tilde g)\to (S^k,g_0)$ of nontrivial $\widehat{A}$--degree and with $\epsilon \cdot \tilde g\geq f^* g_0$ on $T\tilde M$. Then the Poincare dual of $f_!(\widehat{\mathbf{A}}(\tilde M))$ in the top degree is determined by $\deg _{\widehat{A}}f\cdot [S^k]\in H_k(S^k;\Q )$. Since $f$ is proper, we obtain for the standard sphere $S^k_0$
\[
\epsilon \cdot \ke (\tilde M_{\tilde g};\widehat{\mathbf{A}}(\tilde M))\geq \ke (S^k_0;f_!(\widehat{\mathbf{A}}(M))) \geq \ke (S^k_0;\deg _{\widehat{A}}f\cdot [S^k])=\ke (S^k_0).
\]
Because $\epsilon >0$ is arbitrary and $\ke (S^k_0)\in (0,\infty )$, proposition \ref{proposition6} and \ref{proposition7} yield
\[
\ke (M_g;\widehat{\mathbf{A}}(M))=\ke (\tilde{M}_{\tilde {g}};\widehat{\mathbf{A}}(\tilde M))=\infty 
\]
which proves $\widehat{\mathbf{A}}(M)\notin \H ^*(M;\Q )$ (here we used $\pi ^!(\alpha \cap [M])=\pi ^*\alpha \cap [\tilde M]$ for the covering map $\pi :\tilde M\to M$ and a cohomology class on $M$).

The second claim follows similar. Consider a map $f:\tilde M\to S^{2k}$ of nonzero $\widehat{A}$--degree which is constant at infinity. Although $f$ is not proper if $M$ is not compact, the pull back by $f$ yields a map $f^*:\mathscr{V} (S^{2k};\mathbbm{1})\to \mathscr{V}_\mathrm{cpt}(\tilde M;\widehat{\mathbf{A}}(\tilde M))$, here $\mathbbm{1}\in H^0(S^{2k})$ is the unit. This supplies for $\epsilon ^2\cdot \tilde g\geq f^*g_0$ on $\Lambda ^2T\tilde M$:
\[
\epsilon \cdot \ke (\tilde M_{\tilde g};\widehat{\mathbf{A}}(\tilde M))\geq \ke (S^{2k}_0)\in (0,\infty )
\]
and since $\epsilon >0$ is arbitrary, we obtain the claim if $M$ is even dimensional. If $M$ is odd dimensional, consider the composition of the maps
\[
\tilde M\times S^1\stackrel{f\times \mathrm{id}}{\longrightarrow }S^{2k-1}\times S^1\stackrel{\mathrm{pr}}{\longrightarrow} (S^{2k-1}\times S^1)/(S^{2k-1}\vee S^1)\cong S^{2k}
\]
where the one point union $S^{2k-1}\vee S^1$ takes place at a point $(f(x),z)\in S^{2k-1}\times S^1$ and $x \in M$ means a point sufficiently close to infinity. This new map has nonzero $\widehat{A}$--degree and is constant at infinity, i.e.~the even dimensional case provides the claim. 
\end{proof}
If we use the Fredholm K--area introduced by Gromov in \cite{Gr01}, we should be able to generalize the proposition to enlargeable respectively $\widehat{A}$--enlargeable manifolds and to prove $\widehat{\mathbf{A}}(M)\notin \H ^*(M_g;\Q )$ for weakly enlargeable manifolds. The crucial point to this is a version of proposition \ref{proposition7} which is not restricted to finite coverings. Note that the Fredholm K--area can be significantly larger than the K--area because one allows infinite dimensional bundles in the definition.  
\begin{prop}[\cite{GrLa3,Gr01}]
Suppose a spin manifold $M^n$ admits a complete metric $g$ of uniformly positive scalar curvature, then
\[
\ke ( M_g;\widehat{\mathbf{A}}( M))\leq \frac{n(n+1)^3}{2\cdot \inf \, \mathrm{scal}(g)}
\]
which implies $\widehat{\mathbf{A}}(M)\in \H ^*(M_g;\Q )$. 
\end{prop}
Note that the constant on the right hand side does not change for Riemannian covering spaces of $(M,g)$. It might be possible to drop the condition uniformly and to show $\widehat{\mathbf{A}}(M)\in \H ^*(M_g;\Q )$ with a similar argument presented by Gromov and Lawson for the weakly enlargeable case. The proof of the proposition is a simple consequence of the relative index theorem. If $\ke ( M_g;\widehat{\mathbf{A}}(M))$ is not bounded by the constant on the right hand side, we obtain a contradiction to the relative index theorem. The factor $n(n+1)^3/2$ is very rough and includes the estimate of the K--area by the $\mathrm{K}_\mathrm{ch}$--area as well as the estimate of the twisted curvature term in the Bochner formula. If $n$ is even, this factor can be improved to $(n-1)n^3/2$.

Index theory can also be used to show finite K--area of the fundamental homology class on manifolds which are not spin. For instance, suppose that $(M^n,g)$ is a closed connected spin$^c$ manifold with associated class $c\in H^2(M;\Z )$, and let $\Omega $ be a 2--form respresenting the real Chern class of $c$. If $2 | \Omega | _\mathrm{op} <\mathrm{scal}_g$, then
\[
\ke _\mathrm{ch}(M;(e^{c/2}\cdot \widehat{\mathbf{A}}(M))\cap [M])\leq \frac{2n(n-(-1)^n)}{\min (\mathrm{scal}_g-2|\Omega |_\mathrm{op})}\ ,
\]
i.e.~$e^{c/2}\cdot \widehat{\mathbf{A}}(M)\in \H ^*(M;\Q )$ implies $\H _n(M;\Q )=\H ^0(M;\Q )=\Q $. In this case $| \Omega |_\mathrm{op}$ denotes the pointwise operator norm of Clifford multiplication by $\Omega $ which means that for a diagonalization $\Omega =\sum \lambda _je_{2j-1}\wedge e_{2j}$, we have $| \Omega |_\mathrm{op}=\sum |\lambda _j|$. In order to see the above estimate for even $n$, consider the Dirac operator on bundles $\spinor ^cM\otimes \bundle{E}$ where $\spinor ^cM$ is the spin$^c$ bundle and $\bundle{E}\in \mathscr{V}_\mathrm{ch}(M;e^{c/2}\widehat{\mathbf{A}}(M)\cap [M])$ (cf.~\cite{LaMi}). Further examples are obtained by twisting a fixed Dirac bundle $\bundle{S}$ over $M$ with bundles $\bundle{E}$. If $\bundle{S}$ and its connection split locally into $\spinor M\otimes \bundle{X}$ and $4\cdot | \frak{R}^\bundle{X}|_\mathrm{op}<\mathrm{scal}_g$, then the K--areas of certain characteristic classes associated to $\bundle{S}$ are finite. In this case $\frak{R}^\bundle{X}$ is the curvature endomorphism in the Bochner Lichnerowicz formula induced from $\bundle{X}$.
\section{Relative K--area}
We are interested in K--areas for relative homology classes. A particular example is the K--area of the fundamental class of a compact manifold $M$ with nontrivial boundary which gives the total K--area of $M$. In order to simplify the notation, we consider only the coefficient group $\group =\Z $ but most of the statements can be adapted to arbitrary coefficient groups. In this section a \emph{pair of compact manifolds} $(M,A)$ consists of a compact manifold $M$ and a compact submanifold $A\subset M$ which is a strong deformation retract of an open neighborhood in $M$. Both $M$ and $A$ may have nontrivial boundary. Hence, any such pair of compact manifolds satisfies the homotopy extension property which supplies two important facts for us:
\begin{enumerate}
\item[(i)] A continuous map $f:(M,A)\to (N,B)$ is homotopic to a smooth map  $h:(M,A)\to (N,B)$.
\item[(ii)] Let $\bundle{E}\to M$ be a bundle and $M=\coprod M_i$ be a disjoint decomposition in compact manifolds such that $\mathrm{rk}(\bundle{E}_{|M_i})=\mathrm{rk}(\bundle{E}_{|M_j})$ implies $i=j$. Choose points $x_i\in M_i\cap A$ and define the (discrete finite) set $P=\{ x_i\ |\ i\} $. If the restriction of a bundle $\bundle{E}\to M$ to $A \subset M$ is trivial, then there is a classifying map $\rho ^\bundle{E}:M\to \BU $ which is constant on $M_i\cap A$ for all $i$. In particular, $\rho ^\bundle{E}$ induces homomorphisms
\[
\rho ^\bundle{E}_*:H_k(M,A)\to H_k(\BU ,\rho ^\bundle{E}(P) )
\] 
and $\rho ^\bundle{E}(P)\cap \mathrm{BU}_n\subset \BU $ consists of at most one point which means that the inclusion map $H_k(\BU )\to H_k(\BU ,\rho ^\bundle{E}(P))$ is an isomorphism for all $k>0$.
\end{enumerate}
In order to see (i) we apply the smooth approximation theorem to $f_{|A}$ and obtain a smooth map $\tilde h:A\to B$ which is homotopic to $f_{|A}$. Using the homotopy extension property, this homotopy extends to all of $M$ and yields a map $\tilde h:M\to N$ whose restriction to $A $ is smooth and $f\simeq \tilde h:(M,A)\to (N,B)$. Applying the smooth approximation theorem (cf.~\cite[Chp.~II, theorem 11.8]{Bredon}) to $\tilde h$ yields a smooth map $h:(M,A)\to (N,B)$ which is homotopic to $f$. In order to see (ii) let $\tilde \rho ^\bundle{E}:M\to \BU $ be a classifying map, then $\tilde \rho ^\bundle{E}_{\ |A}:A\to \BU $ is a classifying map for the induced bundle on $A\subset M$. Since $\bundle{E}_{|A}$ is trivial, $\tilde \rho ^\bundle{E}_{\ |A}$ is homotopic to a map $\rho ^{\bundle{E}}:A\to  \BU $ which is constant on $M_i\cap A$ for all $i$. Since this homotopy extends to all of $M$, $\tilde \rho ^\bundle{E}$ is homotopic to a map $\rho ^\bundle{E}:M\to \BU $ which is constant on $M_i\cap A$. 

For a pair $(M,A)$ and $\theta \in H_{2*}(M,A)$ respectively $\theta \in H_{2*}(M)$ we denote by $\mathscr{V} (M,A;\theta )$ the set of Hermitian vector bundles $\bundle{E}\to M$ endowed with a Hermitian connection such that $\bundle{E}$ and its connection are trivial on an open neighborhood of $A$ and such that $\rho ^\bundle{E}_*(\theta )\neq 0$. If $\mathscr{V} (M,A;\theta ) $ is not empty, the \emph{relative K--area of} $\theta $ is defined by
\[
\ke (M_g,A;\theta ):=\left( \inf _{\bundle{E}\in \mathscr{V}(M,A;\theta )}\| R^\bundle{E}\| _g \right) ^{-1}\in (0,\infty ].
\]
We set $\ke (M_g,A;\theta )=0$ if $\mathscr{V} (M,A;\theta )=\emptyset $. The relative K--area of an odd class $\theta \in H_{2*+1}(M,A)$ [respectively $\theta \in H_{2*+1}(M)$] is defined by the relative K--area of the class $\theta \times [S^1]\in H_{2*}(M\times S^1,A\times S^1)$ [respectively $\theta \times [S^1]\in H_{2*}(M\times S^1)$]:
\[
\ke (M_g,A;\theta )=\sup _{\mathrm{d}t^2}\ke (M_g\times S^1_{\mathrm{d}t^2},A\times S^1;\theta \times [S^1])
\]
and the relative K--area of a general class is the maximum of the even and the odd part. Note that $A\times S^1 $ is a strong deformation retract of an open neighborhood in $M\times S^1$. Considering the induced homomorphism $j_*:H_k(M )\to H_k(M,A )$ for the inclusion $j:(M,\emptyset )\to (M,A)$ we conclude from the commutative diagram
\[
\begin{xy}
\xymatrix{H_k(M)\ar[d]^{\rho _*^\bundle{E}}\ar[r]^{j_*}&H_k(M,A)\ar[d]^{\rho _*^\bundle{E}}\\
H_k(\BU )\ar[r]^{j_*'}&H_k(\BU ,\rho ^\bundle{E}(P) )}
\end{xy}
\] 
that $\mathscr{V}(M,A;j_*\theta )=\mathscr{V}(M,A;\theta )$ for all $\theta \in H_{k}(M)$ and even $k>0$, here we use that $j'_*$ is an isomorphism for all $k>0$. Hence, the relative K--area of $\theta \in H_k(M)$ and $j_*\theta $ coincide if $k>0$. The case $k=0$ is different because $j_*'$ is not injective. However, for any  $0\neq \theta \in H_0(M,A)$ there is a trivial bundle on $M$ with $\rho ^\bundle{E}(\theta )\neq 0$ which means that $\ke (M_g,A;\theta )=\infty $. Moreover, $\mathscr{V}(M,A;\theta )\subset \mathscr{V}(M;\theta )$ implies:
\[
 \ke (M_g;\theta )\geq \ke (M_g,A;\theta )=\ke (M_g,A;j_*\theta )
\]
for all $\theta \in H_k(M)$ and $k>0$. In some cases there is up to a constant an inequality in the opposite direction as the proposition below shows. We leave it to the reader to verify proposition \ref{proposition1} and \ref{proposition2} for the relative K--area. In particular, if $f:(M,A)\to (N,B)$ is a smooth map with $g\geq f^*h$ on $\Lambda ^2TM$, then the pull back by $f$ yields a map $f^*:\mathscr{V}(N,B;f_*\theta )\to \mathscr{V}(M,A;\theta )$ for even homology classes which implies
\[
\ke (M_g,A;\theta )\geq \ke (N_h,B;f_*\theta )
\]
for all $\theta \in H_*(M,A)$ respectively $\theta \in H_*(M)$. Since $M$ and $A$ are compact, finiteness of the relative K--area does not depend on the Riemannian metric $g$ and we obtain subgroups
\[
\H _k(M,A):=\{ \theta \in H_k(M,A)\ |\ \ke (M_g,A;\theta )<\infty \} .
\]
We have already observed that the K--area of a nontrivial class $\theta \in H_0(M,A)$ is infinite which implies $\H _0(M,A)=\{ 0\} $. Moreover, the above considerations show that the relative and the absolute K--area coincide if $A=\emptyset $, i.e.~$\H _k(M,\emptyset )=\H _k(M)$ for all $k$. Assuming that $f:(M,A)\to (N,B)$ is a continuous map it is homotopic to a smooth map, i.e.~the above inequality for the relative K--area shows
\[
f_*:\H _k(M,A)\to \H _k(N,B).
\]
Furthermore, $\H _*(M,A)$ depends only on the homotopy type of the pair $(M,A)$. Considering rational coefficients and the K--area for the Chern character we can extend the proof of proposition \ref{proposition7} to compact manifolds with boundary and obtain:
\begin{rem}
Let $f:\tilde M_{\tilde g}\to M_g$ be a normal Riemannian covering between oriented compact manifolds with boundary such that $f$ is trivial at the boundary, then
\[
\ke _\mathrm{ch}(\tilde M_{\tilde g},\partial \tilde M;f^!\theta )=\ke _\mathrm{ch}(M_g,\partial M;\theta )
\]
for all $\theta \in H_k(M,\partial M;\Q )$ where $f^!:H_k(M,\partial M;\Q )\to H_k(\tilde M,\partial \tilde M;\Q )$ is the transfer homomorphism defined by the pull back of cohomology classes and Poincare duality using the orientation classes: $f^!=(\cap [\tilde M])\circ f^*\circ (\cap [M])^{-1}$.
\end{rem}

\begin{prop}
Let $(M,A)$ be a pair of compact manifolds and $U\subset M$ be an open set such that $\overline{U}\subset \mathrm{int}(A)$ and $(M\setminus U,A\setminus U)$ is a pair of compact manifolds. Then the inclusion induces isomorphisms
\[
i_*:\H _k(M\setminus U,A\setminus U)\to \H _k(M,A).
\] 
\end{prop}
\begin{proof}
The homomorphisms $i_*$ are well defined by the above consideration and injective by the excision property of singular homology. Moreover, $i^*:\mathscr{V}(M,A;i_*\theta )\to \mathscr{V}(M\setminus U,A\setminus U;\theta )$ is bijective since any bundle $\bundle{E}\to M\setminus U$ which is trivial and has trivial connection on a neighborhood of $A\setminus U$ extends to a bundle $\bundle{E}\to M$ which on a neighborhood of $A\subset M$ is trivial and has  trivial connection. This implies the surjectivity of $i_*$ if $k$ is even. The statement for odd $k$ follows analogous because $(M\setminus U)\times S^1=(M\times S^1)\setminus (U\times S^1)$. 
\end{proof}

\begin{prop}
Suppose $x\in M$, then the inclusion map $j:M\to (M,x)$ yields isomorphisms
\[
j_*:\H _k(M)\to \H _k(M,x)
\]
for all $k$. Moreover, if $(M,A)$ is a pair of compact manifolds such that $M/A$ admits a smooth structure, then
\[
q_*:\H _k(M,A)\to \H _k(M/A,\left< A\right> )\cong \H _k(M/A)
\]
is an isomorphism for all $k$ where $q:M\to M/A$ is the quotient map.
\end{prop}
\begin{proof}
By the above considerations, for each point $x\in M$ the inclusion  $j_*:\H _k(M)\to \H _k(M,x)$ is well defined and injective for all $k$, i.e.~the surjectivity remains to show. In particular, we obtain the claim if there is a constant $C>0$ with
\[
C\cdot \ke  (M_g,x;\theta )\geq \ke (M_g;j_*^{-1}\theta ).
\]
for any $\theta \in H_k(M,x)$ and $k>0$. Let $A=\overline{B_{\epsilon }(x)}\subset M$ be a small closed ball around the point $x$, then $A$ is contractible which means that there is a smooth map $f:M\to M$ which is constant on $A$ but satisfies $f\simeq \mathrm{id}$, i.e.~$f_*=\mathrm{id}$ on homology. Since $f$ is constant on $A$, the pull back by $f$ yields a bundle which has trivial connection on a neighborhood of $x$, i.e.~$f^*:\mathscr{V}(M;\eta )\to \mathscr{V}(M,x;\eta )$. There is a constant $C\geq 1$ with $C\cdot g\geq f^*g$ on $TM$ which implies the curvature estimate for the induced bundle and therefore,
\[
C\cdot \ke  (M_g,x;\eta )\geq \ke (M_g;\eta )
\]
for all $\eta \in H_{2*}(M)$. But the relative K--areas of $\eta \in H_k(M)$ and $j_*\eta \in H_k(M,x)$ coincide for $k>0$ which proves the first claim if $k$ is even. In order to conclude the same estimate for $\eta \in H_{2*+1}(M)$ we use the map $f:M\to M$ from above and consider $b:=f\times \mathrm{id}:M\times S^1\to M\times S^1$. Since $\mathrm{d}b$ has rank one on a neighborhood of $x\times S^1$, the pull back of a bundle by $b$ is flat on a neighborhood of $x\times S^1$ but in general not trivial. Consider the restriction of $b^*\bundle{E}$ to $x\times S^1$, then the connection differs from the trivial connection by a section $\alpha $ in $\mathrm{End}(b^*\bundle{E}_{|x\times S^1})\otimes T^*S^1$. Choosing a cut off function for $B_\epsilon (x)\times S^1\subset M\times S^1$, $\alpha $ extends to a section $M\times S^1\to \mathrm{End}(b^*\bundle{E})\otimes T^*(M\times S^1)$. In fact, $\nabla '=b^*\nabla -\alpha $ is a Hermitian connection on $b^*\bundle{E}$ which on a neighbourhood of $x\times S^1$ is compatible with the trivialization of $b^*\bundle{E}$, i.e.~$(b^*\bundle{E},\nabla ')\in \mathscr{V}(M\times S^1,x\times S^1;\eta \times [S^1])$ if $(\bundle{E},\nabla )\in \mathscr{V}(M\times S^1;\eta \times [S^1])$. Since the $C^0$--bound of $\alpha $ depends on $\bundle{E}$, we need to rescale one of the line elements, i.e.~we fix at first $\mathrm{d}t^2$ and choose $\mathrm{d}s^2=y^2 \cdot \mathrm{d}t^2$ where $y$ satisfies $y\geq \max \{ 1, r\cdot \| \alpha \|  _{\mathrm{d}t^2}\} $ and $r$ means the radius of $S^1_{\mathrm{d}t^2}$ (note that $C\cdot (g\oplus \mathrm{d}s^2)\geq b^*(g\oplus \mathrm{d}t^2)$ on $T(M\times S^1)$). Then for $\bundle{E}\in \mathscr{V}(M\times S^1;\eta \times [S^1])$ the maximum of the curvatures are related by
\[
\| R^{\nabla '}\| _{g\oplus \mathrm{d}s^2}\leq \| R^{b^*\bundle{E}}\| _{g\oplus \mathrm{d}s^2}+C'\cdot \| \alpha \| _{\mathrm{d}s^2}\leq 2\max \left\{ C\cdot \| R^\bundle{E}\| _{g\oplus \mathrm{d}t^2},C'/r\right\}
\]
where $C'<\infty $ depends only on the choice of the cut off function for $B_\epsilon (x)\times S^1\subset M\times S^1$ and $C$ is the constant from above, i.e.~$C$ depends only on $f$. Hence, we conclude 
\[
\sup _{\mathrm{d}s^2}\ke (M_g\times S^1_{\mathrm{d}s^2},x\times S^1;\eta \times [S^1])\geq  \min \left\{ \frac{\ke (M_g\times S^1_{\mathrm{d}t^2} ,\eta \times [S^1])}{2C}, \frac{r}{2C'} \right\} 
\]
which together with the definition for odd homology classes proves the assertion (because $r\to \infty $ by taking the supremum over $\mathrm{d}t^2$ on the right hand side).

Suppose that $A\subset M$ is a submanifold such that $M/A$ is a smooth manifold. $q:(M,A)\to (M/A,\left< A\right> )$ is a relative homeomorphism and an identification map which means that $q_*$ is an isomorphism on the relative singular homology groups. Although $q^*:\mathscr{V}(M/A,\left< A\right> ;q_*\theta )\to \mathscr{V}(M,A;\theta )$ is a bijection, the obvious argument fails in general because we can not choose metrics $g$, $g'$ on $M$ respectively on $M/A$ such that $C\cdot g\geq q^*g'\geq C^{-1}g$ on $M\setminus A$ for some constant $C>0$. However, the K--area homology depends only on the homotopy type of the pair which leads to the following consideration. There is a compact submanifold $X\subset M$ such that $A\subset X$ is a strong deformation retract, then $\left< A\right> \in M/A$ is a strong deformation retract of $Y:=q(X) $, i.e.~$(M,X)\simeq (M,A)$ and $(M/A,Y)\simeq (M/A,\left< A\right> )$. Now we can choose a metric $g'$ on $M/A$ and a corresponding metric $g$ on $M$ such that $g=q^*g'$ on $\overline{M\setminus X}$. The pull back by $q$ yields a bijection $q^*:\mathscr{V}(M/A,Y;q_*\theta )\to \mathscr{V}(M,X;\theta )$ since every bundle which is trivial and flat on a neighborhood of $Y$ extends to a bundle on $M$ that is trivial and flat on a neighborhood of $X$. Thus,  
 $\| R^{q^*\bundle{E}}\| _g=\| R^\bundle{E}\| _{g'}$ proves
\[
\ke (M_g,X;\theta )=\ke  (M/A_{g'},Y;q_*\theta )
\]
for all $\theta \in H_{2*}(M,X)$ which implies that $q_*:\H _{2k}(M,X)\to \H _{2k}(M/A,Y)$ is an isomorphism for all $k$. In order to see this isomorphism for odd homology classes we consider the above setup and the map
\[
q\times \mathrm{id}:(M\times S^1,X\times S^1)\to (M/A\times S^1,Y\times S^1).
\]
By the above argument the relative K--areas of $\theta \times [S^1]$ and $q_*\theta \times [S^1]$ coincide for all $\theta \in H_{2*+1}(M,X)$ which supplies the isomorphism $q_*:\H _k(M,X)\to \H _k(M/A,Y)$. Hence, we obtain the claim from the following commutative diagram
\[
\begin{xy}
\xymatrix{\H _k(M,A)\ar[r]^{q_*}\ar[d]&\H _k(M/A,\left< A\right> )\ar[d]\\
\H _k(M,X)\ar[r]^{q_*}&\H _k(M/A,Y)}
\end{xy}
\]
where the vertical homomorphisms are isomorphisms and induced by the inclusion maps. 
\end{proof}

In general the connecting homomorphism $\partial _*:H_k(M,A )\to H_{k-1}(A )$ does not restrict to a homomorphism $\H _k(M,A)\to \H _{k-1}(A)$ as the following example shows. Let $M=\overline{B^2}$ be the 2--dimensional closed disk and $A=\partial M=S^1$ be its boundary, then
\[
\H _2(M,A)=\H _2(M/A)=\H _2(S^2)=\Z 
\] 
whereas $\H _*(M)=\{ 0\} $ and $\H _*(S^1)=\{ 0\} $, but $\partial _*:H_2(M,A)\to H_1(S^1)$ is an isomorphism. Thus, one may also consider the subgroup of $\H _*(M,A)$ whose image w.r.t.~the connecting homomorphism is contained in $\H _*(A)$. However, for this subgroup excision fails in general. Even if the connecting homomorphism restricts for all $k$ to homomorphisms $\partial _*:\H _k(M,A)\to \H _{k-1}(A)$, the corresponding homology sequence
\[
\cdots \longrightarrow \H _{k}(A)\stackrel{i_*}{\longrightarrow }\H _k(M)\stackrel{j_*}{\longrightarrow }\H _k(M,A)\stackrel{\partial _*}{\longrightarrow }\H _{k-1}(A)\longrightarrow \cdots 
\]
is only a chain complex and not exact in general. In order to see this, let $M=T^n$ be a torus and define $A:=M\setminus B_\epsilon (x)$ for a small open ball $B_\epsilon (x)\subset M$, then $\H _*(M)=\{ 0\}$. Moreover, $M/A\cong S^n$ supplies $\H _*(M,A)=\H _*(S^n)=\Z $ if $n\geq 2$. The homology sequence for the pair $(M,A)$ shows that $i_*:H_k(A)\to H_k(M)$ is injective for all $k<n$ and that the connecting homomorphism is trivial $\partial _*=0$ (the case $k=n-1$ uses the fact that $H_{n-1}(A)$ is torsion free and $M$ closed). Hence, $i_*:\H _k(A)\to \H _k(M)$ is injective and $\H _*(A)=\{ 0\} $ which proves that the homology sequence of this particular pair $(M,A)$ is not exact. 

\begin{defn}
Let $(X,Y)$ be a pair of topological spaces, then $\H _k(X,Y)$ denotes the set of all $\theta \in H_k(X,Y)$ for which there is a pair of compact manifolds $(M,A)$ and a continuous map $f:(M,A)\to (X,Y)$ such that $\theta =f_*(\eta )$ for some $\eta \in \H _k(M,A)$. Furthermore, $\He _k(X,Y)$ is the set of all $\theta \in \H _k(X,Y)$ with $\partial _*\theta \in \H _{k-1}(Y,\emptyset )$ where $\partial _*:H_k(X,Y)\to H_{k-1}(Y,\emptyset )$ is the connecting homomorphism.
\end{defn}
Simple exercises prove that $\He _k(X,Y)\subset \H _k(X,Y)$ are subgroups of $H_k(X,Y)$ for all $k$, that $\H _0(X,Y)=\{ 0\}$ and that $\H (X):=\H (X,Y)=\He (X,Y)$ for $Y=\emptyset $. Moreover, $\H _k(X,Y)$ coincides with the above definition if $(X,Y)$ is a pair of compact manifolds. Suppose that $h:(X,Y)\to (Z,T)$ is a continuous map of pairs, then the induced homomorphism restricts to homomorphisms $h_*:\H _k(X,Y)\to \H _k(Z,T)$ which follows from functoriallity $(h\circ f)_*=h_*f_*$. The naturallity of the connecting homomorphism shows that $h_*$ restricts further to homomorphisms $h_*:\He _k(X,Y)\to \He _k(Z,T)$. Hence, the functors $\H $ and $\He $ satisfy the dimension, additivity and homotopy axiom but obviously not exactness by the examples above. Thus, it remains to consider the excision axiom. Given subspaces $U\subset Y\subset X$ such that $\overline{U}\subset \mathrm{int}(Y)$, then the inclusion map $i:(X\setminus U,Y\setminus U)\to (X,Y)$ yields a well defined injective homomorphism
\[
i_*:\H _k(X\setminus U,Y\setminus U)\to \H _k(X,Y)
\]
for all $k$. We do not know if this homomorphism is surjective in general, however, it is surjective for pairs of compact manifolds and there is much evidence in the case that $(X,Y)$ is a CW--pair. Hence, the above observations can be summarized as follows.
\begin{thm}
The subfunctors $\He \subset \H \subset H$ determine generalized homology theories on the category of pairs of topological spaces and continuous maps which satisfy the dimension, additivity and homotopy axiom. When restricted to the category of pairs of compact manifolds, $\H $ satisfies the excision axiom. Moreover, natural connecting homomorphism exist for both functors and the corresponding homology sequence for pairs is a chain complex: 
\begin{itemize}
\item Choose trivial connecting homomorphisms for $\H $.
\item Choose the connecting homomorphism of singular homology for $\He $. 
\end{itemize} 
\end{thm}

\begin{rem}
Let $\overline{B^m}\subset \R ^m$ be the standard ball with boundary $S^{m-1}$, then $\H _m(\overline{B^m},S^{m-1})=\H _m(S^m) $ by the above proposition. Thus, if $(X,Y)$ is a pair of topological spaces, the relative Hurewicz homomorphism satisfies for all $m\geq 2$:
\[
h_m:\pi _m(X,Y)\to \H _m(X,Y)\subset H_m(X,Y).
\]
If $m\geq 3$, $\He _m(\overline{B^m},S^{m-1})=\H _m(\overline{B^m},S^{m-1})$ supplies that $\mathrm{Im}(h_m)\subset \He _m(X,Y)$.
\end{rem}
\begin{rem}
The K--area homology $\H $ stabalizes for compact manifolds under the suspension operation into reduced singular homology. In order to see this let $M$ be a compact manifold, $\Sigma ^kM=S^k\wedge M$ be the $k$--fold reduced suspension of $M$ and consider the projection map $p:S^k\times M\to \Sigma ^kM $. Suppose that $\eta \in H_{k+i}(\Sigma ^kM)$, then $\eta =p_*([S^k]\times \theta )$ for some $\theta \in H_i(M)$, but $[S^k]\times \theta $ has finite K--area for all $k\geq 2$ (cf.~remark \ref{rem8}) which proves that $\eta \in \H _{k+i}(\Sigma ^kM)$.  Hence, $\H _*(\Sigma ^kM)$ is the reduced singular homology of $\Sigma ^kM$ for all $k\geq 2$. If $M$ is connected, this may still be true for $k=1$, however, it fails in general as the example $\Sigma ^1S^0=S^1$ shows. Note that $\Sigma M$ is simply connected for connected manifolds and that $\H _j(\Sigma T^n)=H_j(\Sigma T^n)$ for all $j>0$, $n>0$ by an induction argument: $\Sigma S^1=S^2$ as well as $\Sigma T^n\simeq S^2\vee \Sigma T^{n-1}\vee \Sigma ^2T^{n-1}$.  
\end{rem}

In the following we give an example of manifolds with isomorphic fundamental group and isomorphic singular homology, but with different K--area homology. These manifolds have different intersection product, but it seems rather difficult to construct manifolds with isomorphic cohomology ring and compute their K--area homology. Nevertheless, we expect the K--area homology to collect additional topological data. Suppose $n>3$ and define
\[
\begin{split}
M&=(T^3\times S^n)\# (T^3\times S^n)\\
N&=\bigl( (T^3\# T^3)\times S^n\bigl) \# (S^3\times S^n),
\end{split}
\]
then $\pi _1(M)=\pi _1(N)=\Z ^3*\Z ^3$ by Seifert--van Kampen. We leave it to the reader to compute the singular homology. In order to determine the K--area homology of $M$ and $N$ we use the methods presented in remark \ref{zusammen} and results about positive scalar curvature on closed spin manifolds, in fact we obtain
\[
\begin{split}
\H _k(M)&=\H _k(N)=H_k(M)=H_k(N)\quad  k\geq 4\\
\H _j(M)&=\H _j(N)=\{ 0\}\qquad \ \ j=0,1,2
\end{split}
\]
whereas $\H _3(N)=\Z $ and $\H _3(M)=\{ 0\} $.

\section*{Acknowlegdements}
The author would like to thank Sebastian Goette and Jan Schl\"uter for their support in questions of topology.

\bibliographystyle{abbrv}
\bibliography{/home/mlisting/bib/bibliothek}

\end{document}